\documentclass[a4paper,11pt]{amsart}
\usepackage{amssymb,amsmath,amsfonts}
\usepackage{graphicx}
\usepackage{latexsym,mathrsfs}

\newtheorem{theorem}{Theorem}[section]

\newtheorem{proposition}[theorem]{Proposition}
\newtheorem{corollary}[theorem]{Corollary}

\newcommand{\Be}{\begin{equation}}
\newcommand{\Ee}{\end{equation}}
\newcommand{\Bea}{\begin{align}}
\newcommand{\Eea}{\end{align}}
\newcommand{\Beas}{\begin{align*}}
\newcommand{\Eeas}{\end{align*}}
\newcommand{\Bes}{\begin{equation*}}
\newcommand{\Ees}{\end{equation*}}

\title[Sharp weighted norm estimates for the Bergman operator]{Sharp weighted norm estimates for the positive Bergman operator: the two-parameter case}
\author[B. F. Sehba]{Benoit F. Sehba}
\address{Department of Mathematics, University of Ghana}
\email{bfsehba@ug.edu.gh}

\begin{document}

\begin{abstract}
We  allow the exponent $\alpha$ that defines the positive Bergman operator on the upper half-plane to differ from the exponent $\gamma$ that defines the underlying weighted Lebesgue spaces. In this setting, we prove a sharp weak-type off-diagonal weighted inequality, with explicit and best-possible dependence on the B\'ekoll\'e--Bonami characteristic of the weight. We then deduce strong-type off-diagonal inequalities for the same operator. The key new tool is an off-diagonal extrapolation theorem, adapted to allow a change in the power of the distance to the boundary along the way.
\end{abstract}

\maketitle
\section{Introduction}

A central theme in modern harmonic analysis is the quantification of weighted operator norms in terms of the characteristic of the weight. For the Hilbert transform and general Calder\'on--Zygmund operators, this question was resolved by Hyt\"onen's celebrated $A_2$ theorem \cite{Hytonen} (see also \cite{Cruzetal,DGPP,HyLaPerez,HyPerez,Laceyetal,Laceyetal2,Lerner,Petermichl,PetVol}), and it has since become a template for a wide range of singular and non-singular operators. The Bergman projection is a natural test case in this context, it is fundamental in complex and harmonic analysis, admits a genuine weighted theory going back to B\'ekoll\'e and Bonami, and yet its sharp quantitative behaviour was understood only recently. Our purpose here is to push this sharp theory into a previously unexplored setting in which the projection and the weighted space on which it acts are governed by distinct exponents. This extra flexibility, as we explain below, is precisely what is needed to complete the off-diagonal theory, and supplying it is the main objective of this note.

Let $\mathbb{R}_+^2:=\{z\in\mathbb{C}: \operatorname{Im} z>0\}$ denote the upper half-plane. For $1\le p<\infty$ and $\alpha>-1$, write $dV_\alpha(x+iy)=y^\alpha\,dx\,dy$, and let $L^p(\omega\, dV_\alpha)$ denote the space of functions $f$ on $\mathbb{R}_+^2$ with
\[
\|f\|_{p,\omega,\alpha}^p:=\int_{\mathbb{R}_+^2}|f(z)|^p\omega(z)\,dV_\alpha(z)<\infty.
\]
When $\omega\equiv 1$ we write $L^p(dV_\alpha)$. The Bergman projection $P_\alpha$ is the orthogonal projection of $L^2(dV_\alpha)$ onto its closed subspace of holomorphic functions,
\[
P_{\alpha}f(z):=\int_{\mathbb{R}_+^2}\frac{f(w)}{(z-\overline{w})^{2+\alpha}}\,dV_\alpha(w),
\]
and its positive counterpart is
\[
P_{\alpha}^+f(z):=\int_{\mathbb{R}_+^2}\frac{f(w)}{|z-\overline{w}|^{2+\alpha}}\,dV_\alpha(w).
\]
Both operators are bounded on $L^p(dV_\alpha)$ for every $1<p<\infty$. The subtle question, however, is for which weights $\omega$ they remain bounded on $L^p(\omega\,dV_\alpha)$, and with what precise dependence of the operator norm on the weight.

The qualitative answer was given by B\'ekoll\'e and Bonami. Let $\alpha>-1$ and $1<p<\infty$, define $p'$ by $\frac{1}{p}+\frac{1}{p'}=1$. For an interval $I\subset\mathbb{R}$, let $Q_I:=\{z=x+iy: x\in I,\ 0<y<|I|\}$ be the associated Carleson box, and put $|Q_I|_\alpha=V_\alpha(Q_I)=\int_{Q_I}dV_\alpha$. The B\'ekoll\'e--Bonami class $B_{p,\alpha}$ consists of weights $\omega$ for which
\[
[\omega]_{B_{p,\alpha}}:=\sup_{I}\left(\frac{1}{|Q_I|_\alpha}\int_{Q_I}\omega\,dV_\alpha\right)\left(\frac{1}{|Q_I|_\alpha}\int_{Q_I}\omega^{1-p'}\,dV_\alpha\right)^{p-1}<\infty,
\]
with the obvious endpoint modification for $B_{1,\alpha}$. By \cite{Bek,BB}, for $1<p<\infty$, $P_\alpha^+$ is bounded on $L^p(\omega\,dV_\alpha)$ exactly when $\omega\in B_{p,\alpha}$, and it maps $L^1(\omega\,dV_\alpha)$ into weak-$L^1(\omega\,dV_\alpha)$ exactly when $\omega\in B_{1,\alpha}$. Once this qualitative boundedness is known, Hyt\"onen's work \cite{Hytonen} made it natural to seek the sharp quantitative version. Pott and Reguera answered this in \cite{PR}: for $1<p<\infty$,
\Be\label{eq:PottReg}
\|P_\alpha^+\|_{L^p(\omega\, dV_\alpha)\rightarrow L^p(\omega\, dV_\alpha)}
\le c[\omega]_{B_{p,\alpha}}^{\max\{1,\frac{p'}{p}\}},
\Ee
and the exponent is best possible. This diagonal  result is the starting point of the entire sharp theory for the Bergman projection.

Diagonal estimates such as \eqref{eq:PottReg} are only half the story. Off-diagonal estimates, i.e. bounding an operator from $L^p$ into a different $L^q$, with $q>p$, are the natural framework for smoothing operators. To access this off-diagonal world for the Bergman projection, the fractional Bergman operator was introduced in \cite{Sehba4}: for $\alpha>-1$ and $0\le \delta<2+\alpha$,
\Be\label{eq:fracBergdef}
T_{\alpha,\delta}f(z):=\int_{\mathbb{R}_+^2}\frac{f(w)}{|z-\overline{w}|^{2+\alpha-\delta}}\,dV_\alpha(w), \qquad T_{\alpha,0}=P_\alpha^+,
\Ee
which is the half-plane analogue of the classical fractional integral. Its natural weight classes are the two-exponent B\'ekoll\'e--Bonami classes $B_{p,q,\alpha}$. For $1<p,q<\infty$, and $\alpha>-1$, we say a weight $\omega$ belongs to $B_{p,q,\alpha}$ if
\[
[\omega]_{B_{p,q,\alpha}}:=\sup_{I}\left(\frac{1}{|Q_I|_{\alpha}}\int_{Q_I}\omega^q\,dV_{\alpha}\right)
\left(\frac{1}{|Q_I|_{\alpha}}\int_{Q_I}\omega^{-p'}\,dV_{\alpha}\right)^{q/p'}<\infty,
\]
with the corresponding endpoint class $B_{1,q,\alpha}$. It was shown in \cite{Sehba5} that, with $\frac{1}{q}=\frac{1}{p}-\frac{\delta}{2+\alpha}$, membership in $B_{p,q,\alpha}$ characterizes boundedness of $T_{\alpha,\delta}:L^p(\omega^p\,dV_\alpha)\to L^q(\omega^q\,dV_\alpha)$, and the sharp exponent was subsequently found in \cite{Sehba6}: for $1<p<\frac{2+\alpha}{\delta}$,
\Be\label{eq:Sehbasharp1}
\|T_{\alpha,\delta}\|_{L^p(\omega^p\,dV_\alpha)\rightarrow L^q(\omega^q\,dV_\alpha)}
\le c[\omega]_{B_{p,q,\alpha}}^{\left(1-\frac{\delta}{2+\alpha}\right)\max\{1,\frac{p'}{q}\}}.
\Ee

The fractional-operator route, however, forces the exponent $\alpha$ defining the operator to coincide with the exponent governing the ambient weighted spaces. This coupling becomes a genuine obstruction when one tries to turn \eqref{eq:Sehbasharp1} into a genuine off-diagonal statement for $P_\alpha^+$ itself. Indeed, using the pointwise bound $y^\delta|P_\alpha^+f(z)|\le |T_{\alpha,\delta}f(z)|$, and choosing $\delta=\frac{\beta-\alpha}{q}$ turns \eqref{eq:Sehbasharp1} into a sharp off-diagonal strong-type bound for $P_\alpha^+$ between weighted spaces of different exponents $\alpha$ and $\beta$,
\Be\label{eq:strongberg}\| {P}_{\alpha}^+\|_{L^p(\omega^p\,\mathrm{d}V_{\alpha})\longrightarrow L^q(\omega^q\,\mathrm{d}V_{\beta})}\lesssim [\omega]_{B_{p,q,\alpha}}^{\left(\frac{1}{p'}+\frac{1}{q}\right)\max\{1,\frac{p'}{q}\}}, \qquad \beta=(2+\alpha)\Big(\frac{q}{p}-1\Big) + \alpha.\Ee
But the weak-type analogue of \eqref{eq:Sehbasharp1} from \cite{Sehba6},
\begin{equation}\label{eq:sharpweak}\| T_{\alpha,\delta}f\|_{L^{q,\infty}(\omega^{q}\,\mathrm{d}V_{\alpha})}\le c[\omega]_{B_{p,q,\alpha}}^{1-\frac{\delta}{2+\alpha}}\|\omega f\|_{p,\alpha},\end{equation}
degenerates at $\delta=0$ to nothing more than the classical weak-type bound
\begin{equation}\label{eq:sharpweakberg}\| P_{\alpha}^+f\|_{L^{p,\infty}(\omega^p\,\mathrm{d}V_{\alpha})}\le c[\omega^p]_{B_{p,\alpha}}\|f\|_{p,\omega^p,\alpha}.\end{equation}
 Thus, while the strong-type off-diagonal theory for $P_\alpha^+$ was accessible through $T_{\alpha,\delta}$, its natural weak-type companion was not. This asymmetry is the starting point of the present note.

The fractional-operator approach intrinsically ties the Bergman exponent $\alpha$ to the measure exponent $\gamma$, requiring the same parameter to govern both the operator and the underlying weighted spaces. The central objective of this paper is to break this coupling. In this note we allow the exponent $\alpha$ (defining the operator) and the exponent $\gamma$ (defining the domain space) to be genuinely distinct, with $\alpha\ge\gamma$. We establish the sharp weak-type off-diagonal theory for $P_\alpha^+$ in this two-parameter setting and, as an application, derive a corresponding strong-type estimate.

\begin{theorem}\label{thm:sharpfracbergnew}
Let $\alpha\geq \gamma>-1$ and $1\le p\le q<\infty$. Define $\beta=(2+\gamma)(q/p-1) + \gamma$, and assume $\omega\in B_{p,q,\gamma}$. Then
\begin{equation}\label{eq:sharpfracbergnew}
\| P_{\alpha}^+\|_{L^p(\omega^p\,dV_{\gamma})\longrightarrow L^{q,\infty}(\omega^q\,dV_{\beta})}
\lesssim [\omega]_{B_{p,q,\gamma}}^{1/p'+1/q}.
\end{equation}
Moreover, this estimate is sharp.
\end{theorem}

\begin{theorem}\label{thm:sharpweakbergnew}
Let $\alpha, \gamma>-1$ and $1<p\le q<\infty$. Define $\beta=(2+\gamma)(q/p-1) + \gamma$. Assume $\omega\in B_{p,q,\gamma}$ and $1\le\frac{2+\alpha}{2+\gamma}\le\frac{q}{p}$. Then
\Be\label{eq:sharpbergnew}
\| {P}_{\alpha}^+\|_{L^p(\omega^p\,dV_{\gamma})\longrightarrow L^q(\omega^q\,dV_{\beta})} \lesssim [\omega]_{B_{p,q,\gamma}}^{\frac{1}{p'}+\frac{1}{q}}+[\omega]_{B_{p,q,\,\beta+(\alpha-\beta)q'}}^{\left(\frac{1}{p'}+\frac{1}{q}\right)\frac{p'}{q}}.
\Ee
\end{theorem}
The condition $1\le\frac{2+\alpha}{2+\gamma}\le\frac{q}{p}$ is equivalent to $\alpha\ge \gamma$ and $\alpha\ge\beta+(\alpha-\beta)q'$ (or equivalently, $\beta\ge\alpha$). The inequality $\alpha\ge \gamma$ is needed for the estimate (\ref{eq:sharpfracbergnew}), while $\alpha\ge\beta+(\alpha-\beta)q'$ is needed for the dual estimate. The inequality (\ref{eq:sharpbergnew}) follows by adding these two weak-type estimates.

We derive the following strong-type estimate.
\begin{corollary}\label{cor:sharpstrong}
Let $\alpha, \gamma>-1$ and $1<p\le q<\infty$. Assume $\omega\in B_{p,q,\gamma}$ and $\frac{2+\alpha}{2+\gamma}=\frac{1}{p}+\frac{1}{q'}$. Define $\beta=(2+\gamma)(q/p-1) + \gamma$. Then
\Be\label{eq:sharpstrong}
\| {P}_{\alpha}^+\|_{L^p(\omega^p\,dV_{\gamma})\longrightarrow L^q(\omega^q\,dV_{\beta})}\lesssim  [\omega]_{B_{p,q,\gamma}}^{\left(\frac{1}{p'}+\frac{1}{q}\right)\max\{1,\frac{p'}{q}\}}.
\Ee
Moreover, this is sharp.
\end{corollary}
Note that as $1<p\le q<\infty$ and $1+\left(\frac{1}{p}-\frac{1}{q}\right)\ge 1$, the relation $\frac{2+\alpha}{2+\gamma}=\frac{1}{p}+\frac{1}{q'}$ implies that $\alpha\ge \gamma$ and $\beta\ge\alpha$. The estimate follows by observing that in this case, $\beta+(\alpha-\beta)q'=\gamma$.
\vskip .3cm
Several features of these results deserve emphasis. First, the exponent governing the growth of the operator norm in the weak-type estimate is completely explicit and independent of $\alpha$; the role of $\alpha$ is confined to the implicit constant. Second, the weight characteristic that controls everything, $[\omega]_{B_{p,q,\gamma}}$, is built entirely from the domain parameter $\gamma$, the operator's own parameter $\alpha$ never appears inside the weight class. This is precisely the extra freedom that the previous diagonal and fractional-operator approaches could not deliver. Third,  for the strong-type inequality, we obtain a sum of two characteristics; one characteristic comes from the direct weak-type estimate while the other is from the dual estimate. The two characteristics in (\ref{eq:sharpbergnew}) coincide only when $\gamma=\beta+(\alpha-\beta)q'$ or equivalently, when the balance condition $\frac{2+\alpha}{2+\gamma}=\frac{1}{p}+\frac{1}{q'}$ holds. Our guess is that the estimate (\ref{eq:sharpbergnew}) cannot be reduced into a single characteristic in general.

The key tool behind both theorems is a new extrapolation principle. Rubio de Francia extrapolation, and its sharp quantitative descendants developed for the Bergman setting in \cite{Laceyetal,Sehba6}, allow one to propagate a sharp weighted estimate from a single pair of exponents $(p_0,q_0)$ to the entire admissible scale, while keeping track of the dependence on the weight characteristic. What was missing from the existing statement is the ability to let the power of the distance to the boundary move together with the exponents, earlier extrapolation kept this power frozen at $\alpha$. We remove that restriction 
by first renormalising the operator with a suitable power of $y=\operatorname{Im} z$ to reduce to the fixed-parameter statement, applying extrapolation, and then undoing the renormalisation. This is precisely what allows the exponents $\beta$ and $\gamma$ in Theorems \ref{thm:sharpfracbergnew} to be genuinely decoupled from $\alpha$. For the sharpness half of Theorem \ref{thm:sharpfracbergnew}, extrapolation alone is not enough; we exhibit an explicit family of testing weights and functions that simultaneously saturates the inequality, balancing all parameters $p,q,\alpha,\beta,\gamma$ against one another.

The presentation follows the general strategy of \cite{Sehba6}, adapted throughout to accommodate the extra parameter; the paper is otherwise self-contained modulo the Sawyer-type testing characterizations of $P_\alpha^+=T_{\alpha,0}$ established there, which we recall as needed.

\section{Proofs of the results}
 Let us denote by $\mathcal{I}$ the set of all subintervals of $\mathbb{R}$. For $\gamma>-1$,  the  Hardy-Littlewood maximal function for the upper half-plane associated to $\gamma$, $\mathcal{M}_\gamma$ is given by
\[
\mathcal{M}_\gamma f(z):=\sup_{I\in\mathcal{I}}\frac{\chi_{Q_I}(z)}{\vert Q_I\vert_\gamma}\int_{Q_I}\vert f(\zeta)\vert dV_\gamma(\zeta).
\]
\subsection{Proof of Theorem \ref{thm:sharpfracbergnew}}
We first record an adapted sharp extrapolation result, needed for the proof of our main results. It follows from a suitable reinterpretation of \cite[Theorem 6]{Sehba6}, which we recall below.
\begin{theorem}\label{thm:Sehba6}
Let $T$ be an operator defined on a suitable class (e.g. $C_c^{\infty}$ or $\displaystyle\bigcup_p L^p(\omega^p\,\mathrm{d}V_{\alpha})$, $\alpha>-1$). Suppose $p_0,q_0$ are exponents with $1\le p_0\le q_0<\infty$ such that
\[
\|\omega Tf\|_{q_0,\alpha}\le c[\omega]^{\theta}_{B_{p_0,q_0,\alpha}}\|\omega f\|_{p_0,\alpha}
\]
holds for all $\omega\in B_{p_0,q_0,\alpha}$ and some $\theta>0$. Then
\[
\|\omega Tf\|_{q,\alpha}\le c[\omega]^{\theta\max\{1,\frac{q_0}{p'_0}\cdot\frac{p'}{q}\}}_{B_{p,q,\alpha}}\|\omega f\|_{p,\alpha}
\]
for all $p,q$ satisfying $1<p\le q<\infty$ and $\frac{1}{p}-\frac{1}{q}=\frac{1}{p_0}-\frac{1}{q_0}$, and for all $\omega\in B_{p,q,\alpha}$.
\end{theorem}

From this we deduce the following.
\begin{theorem}\label{thm:sharpextra}
Let $T$ be an operator defined on a suitable class (e.g. $C_c^{\infty}$ or $\displaystyle\bigcup_p L^p(\omega^p\,\mathrm{d}V_{\alpha})$, $\alpha>-1$). Suppose $p_0,q_0$ are exponents and $\beta_0$ is a real number, with $1\le p_0\le q_0<\infty$, $\frac{2+\beta_0}{q_0}= \frac{2+\alpha}{p_0}$, and
\[
\|\omega Tf\|_{q_0,\beta_0}\le c[\omega]^{\theta}_{B_{p_0,q_0,\alpha}}\|\omega f\|_{p_0,\alpha}
\]
for all $\omega\in B_{p_0,q_0,\alpha}$ and some $\theta>0$. Then
\[
\|\omega Tf\|_{q,\beta}\le c[\omega]^{\theta\max\{1,\frac{q_0}{p'_0}\cdot\frac{p'}{q}\}}_{B_{p,q,\alpha}}\|\omega f\|_{p,\alpha}
\]
for all $p,q,\beta$ satisfying $1<p\le q<\infty$, $\frac{1}{p} - \frac{1}{q} = \frac{1}{p_0} - \frac{1}{q_0}$, and $\frac{2+\beta}{q} = \frac{2+\alpha}{p}$, and for all $\omega\in B_{p,q,\alpha}$.
\end{theorem}
\begin{proof}
Let $p_0,q_0,\beta_0$ be as in the statement, and suppose
\[
\|\omega Tf\|_{q_0,\beta_0}\le c[\omega]^{\theta}_{B_{p_0,q_0,\alpha}}\|\omega f\|_{p_0,\alpha}
\]
holds for all $\omega\in B_{p_0,q_0,\alpha}$ and some $\theta>0$. Set $Sf(x+iy)=y^{(2+\alpha)\left(\frac{1}{p_0}-\frac{1}{q_0}\right)}Tf(x+iy)$. Using the balance relation
$\frac{2+\beta_0}{q_0}= \frac{2+\alpha}{p_0}$, we obtain
\[
\|\omega Sf\|_{q_0,\alpha}\le c[\omega]^{\theta}_{B_{p_0,q_0,\alpha}}\|\omega f\|_{p_0,\alpha}
\]
for all $\omega\in B_{p_0,q_0,\alpha}$ and the same $\theta>0$. Theorem \ref{thm:Sehba6} then gives
\[
\|\omega Sf\|_{q,\alpha}\le c[\omega]^{\theta\max\{1,\frac{q_0}{p'_0}\cdot\frac{p'}{q}\}}_{B_{p,q,\alpha}}\|\omega f\|_{p,\alpha}
\]
for all $p,q,\beta$ satisfying $1<p\le q<\infty$, $\frac{1}{p}-\frac{1}{q} = \frac{1}{p_0}-\frac{1}{q_0}$, and all $\omega\in B_{p,q,\alpha}$.
Finally, using the balance relation $\frac{2+\beta}{q}=\frac{2+\alpha}{p}$ together with the definition of $S$, we obtain the conclusion of the theorem.
\end{proof}

We now deduce the following weak-type extrapolation result.
\begin{corollary}\label{cor:weaksharpextra}
Suppose that for some $1\le p_0\le q_0<\infty$ and $\frac{2+\beta_0}{q_0}= \frac{2+\alpha}{p_0}$, the operator $T$ satisfies the weak-type $(p_0,q_0)$ inequality
\[
\| Tf\|_{L^{q_0,\infty}(\omega^{q_0}\,\mathrm{d}V_{\beta_0})}\le c[\omega]_{B_{p_0,q_0,\alpha}}^{\theta}\|\omega f\|_{p_0,\alpha}
\]
for every $\omega\in B_{p_0,q_0,\alpha}$ and some $\theta>0$. Then $T$ also satisfies the weak-type $(p,q)$ inequality
\[
\| Tf\|_{L^{q,\infty}(\omega^{q}\,\mathrm{d}V_{\beta})}\le c[\omega]_{B_{p,q,\alpha}}^{\theta\max\{1,\frac{q_0}{p'_0}\cdot\frac{p'}{q}\}}\|\omega f\|_{p,\alpha}
\]
for all $1<p\le q<\infty$ satisfying $\frac{1}{p}-\frac{1}{q} = \frac{1}{p_0} - \frac{1}{q_0}$ and $\frac{2+\beta}{q} = \frac{2+\alpha}{p}$, and for all $\omega\in B_{p,q,\alpha}$.
\end{corollary}

\begin{proof}
Set $T_{\lambda}f = \lambda\chi_{\{|Tf|>\lambda\}}$. For fixed $\lambda>0$, applying Theorem \ref{thm:sharpextra} to $T_{\lambda}$, we find a constant $c$, independent of $\lambda$, such that
\begin{align*}
\|\omega T_{\lambda}f\|_{q_0,\beta_0}
&= \lambda|\{z\in\mathbb{R}_+^2:|Tf(z)|>\lambda\}|_{\omega^{q_0},\beta_0}^{\frac{1}{q_0}}\\
&\le \|Tf\|_{L^{q_0,\infty}(\omega^{q_0}\,\mathrm{d}V_{\beta_0})}\\
&\le c[\omega]_{B_{p_0,q_0,\alpha}}^{\theta}\|\omega f\|_{p_0,\alpha}.
\end{align*}
Thus, if $\omega\in B_{p,q,\alpha}$ and $\frac{1}{p} - \frac{1}{q} = \frac{1}{p_0} - \frac{1}{q_0}$, Theorem \ref{thm:sharpextra} shows that
\[
T_{\lambda}: L^p(\omega^p\,\mathrm{d}V_{\alpha})\longrightarrow L^q(\omega^q\,\mathrm{d}V_{\beta})
\]
is bounded, with
\[
\| \omega T_{\lambda}f\|_{q,\beta}\le c[\omega]_{B_{p,q,\alpha}}^{\theta\max\{1,\frac{q_0}{p'_0}\cdot\frac{p'}{q}\}}\|\omega f\|_{p,\alpha},
\]
where $c$ does not depend on $\lambda$. Hence
\begin{align*}
\| Tf\|_{L^{q,\infty}(\omega^{q}\,\mathrm{d}V_{\beta})}
&= \sup_{\lambda>0}\|\omega T_{\lambda}f\|_{q,\beta}\\
&\le c[\omega]_{B_{p,q,\alpha}}^{\theta\max\{1,\frac{q_0}{p'_0}\cdot\frac{p'}{q}\}}\|\omega f\|_{p,\alpha}.
\end{align*}
\end{proof}

To prove Theorem \ref{thm:sharpweakbergnew} we first establish the following estimate.
\begin{proposition}\label{prop:weak initial}
Let $-1<\gamma\le \alpha<\infty$, $\beta_0>-1$, and define $q_0$ by $\frac{2+\beta_0}{q_0}=2+\gamma$. Then, for any weight $u$,
\Be\label{eq:Bermaxbound}\| P_{\alpha}^+f\|_{L^{q_0,\infty}(u\,\mathrm{d}V_{\beta_0})}\le c\|f\|_{1,(\mathcal{M}_{\gamma}u)^{\frac{1}{q_0}},\gamma}.\Ee
\end{proposition}
\begin{proof}
We proceed as in \cite[Proposition 2]{Sehba6}. By \cite[Theorem 6.3]{BKS}, we have
\[
\| P_{\alpha}^+f\|_{L^{q_0,\infty}(u\,\mathrm{d}V_{\beta_0})}\le c_{q_0} \int_{\mathbb{R}^2_+}|f(z)|\,\||\cdot - \bar{z}|^{-2-\alpha}\|_{L^{q_0,\infty}(u\,\mathrm{d}V_{\beta_0})}\,\mathrm{d}V_{\alpha}(z).
\]
We now estimate $\||\cdot - \bar{z}|^{-2-\alpha}\|_{L^{q_0,\infty}(u\,\mathrm{d}V_{\beta_0})}$.

For any $t>0$,
\begin{align*}
A & := t|\{w\in\mathbb{R}^2_+ : |w-\bar{z}|^{-2-\alpha} > t\}|^{\frac{1}{q_0}}_{u,\beta_0}\\
&= t\left|\left\{w\in\mathbb{R}^2_+ : |w-\bar{z}| <\left(\frac{1}{t}\right)^{\frac{1}{2+\alpha}}\right\}\right|^{\frac{1}{q_0}}_{u,\beta_0}\\
&=\frac{1}{\lambda^{2+\alpha}}|\{w\in\mathbb{R}^2_+ : |w-\bar{z}| <\lambda \}|^{\frac{1}{q_0}}_{u,\beta_0}.
\end{align*}
Since $\{w\in\mathbb{R}^2_+ : |w-\bar{z}| <\lambda \}$ is contained in the Carleson box $Q_I$, where $I$ is the interval centered at $\operatorname{Re} z$ with length $|I|=\lambda$, this observation together with the relation $\beta_0=(2+\gamma)(q_0-1)+\gamma$ gives
\begin{align*}
A &= \frac{1}{\lambda^{2+\alpha}}|\{w\in\mathbb{R}^2_+ : |w-\bar{z}| <\lambda \}|^{\frac{1}{q_0}}_{u,\beta_0}\\
&\le \frac{1}{\lambda^{2+\alpha}}|Q_I|^{\frac{1}{q_0}}_{u,\beta_0}\\
&\le \frac{|I|^{\frac{\beta_0-\gamma}{q_0}}}{\lambda^{2+\alpha}}|Q_I|^{\frac{1}{q_0}}_{u,\gamma}\\
&\le |I|^{\gamma-\alpha}\left( \frac{1}{|I|^{2+\gamma}}|Q_I|_{u,\gamma}\right)^{\frac{1}{q_0}}\\
&\le \left(\operatorname{Im} z\right)^{\gamma-\alpha}\sup_I\left( \frac{1}{|I|^{2+\gamma}}|Q_I|_{u,\gamma}\right)^{\frac{1}{q_0}}\\
&= \left(\operatorname{Im} z\right)^{\gamma-\alpha}\left( \mathcal{M}_{\gamma}u\right)^{\frac{1}{q_0}}.
\end{align*}
Thus
\begin{align*}
\||\cdot - \bar{z}|^{-2-\alpha}\|_{L^{q_0,\infty}(u\,\mathrm{d}V_{\beta_0})}
&:= \sup_{t>0} t|\{w\in\mathbb{R}^2_+ : |w-\bar{z}|^{-2-\alpha} > t\}|^{\frac{1}{q_0}}_{u,\beta_0}\\
&\le \left(\operatorname{Im} z\right)^{\gamma-\alpha}\left( \mathcal{M}_{\gamma}u\right)^{\frac{1}{q_0}}.
\end{align*}
Consequently,
\[
\| P_{\alpha}^+f\|_{L^{q_0,\infty}(u\,\mathrm{d}V_{\beta_0})}\le c_{q_0} \int_{\mathbb{R}^2_+}|f(z)|(\mathcal{M}_{\gamma}u)^{\frac{1}{q_0}}\,\mathrm{d}V_{\gamma}.
\]
This completes the proof.
\end{proof}
From the above proposition we deduce the following.
\begin{corollary}\label{cor:weak initialestim}
Let $-1<\gamma\leq \alpha$, $\beta_0>-1$, and define $q_0$ by $\frac{2+\beta_0}{q_0}=2+\gamma$. Then
\Be\label{eq:weightmax}\| P_{\alpha}^+f\|_{L^{q_0,\infty}(\omega^{q_0}\,{d}V_{\beta_0})}\le c[\omega]_{B_{1,q_0,\gamma}}^{\frac{1}{q_0}}\|f\|_{1,\omega,\gamma}.\Ee
\end{corollary}
\begin{proof}
If $u=\omega^{q_0}$, then \[
\mathcal{M}_{\gamma}u\le c[u]_{B_{1,\gamma}}u=c[\omega]_{B_{1,q_0,\gamma}}u.
\]
Substituting into \eqref{eq:Bermaxbound} gives \eqref{eq:weightmax}.
\end{proof}
We are now ready to prove our first main result.
\begin{proof}[Proof of Theorem \ref{thm:sharpfracbergnew}]
Let $\alpha\geq \gamma>-1$ and $1\leq p\leq q<\infty$. Define $\beta$ by $\frac{2+\beta}{q}=\frac{2+\gamma}{p}$. Define $\beta_0$ by $\frac{2+\gamma}{2+\beta_0}=\frac{1}{p'}+\frac{1}{q}$, and $q_0>1$ by $\frac{2+\beta_0}{q_0}=2+\gamma$. Put $p_0=1$. By Corollary \ref{cor:weak initialestim}, for any $\omega\in B_{p_0,q_0,\gamma}$,
\[
\|P_{\alpha}^+f\|_{L^{q_0,\infty}(\omega^{q_0}\,{d}V_{\beta_0})}\le c[\omega]_{B_{p_0,q_0,\gamma}}^{\frac{1}{q_0}}\|f\|_{p_0,\omega^{p_0},\gamma}.
\]
Since $\frac{1}{p}-\frac{1}{q}=\frac{1}{p_0}-\frac{1}{q_0}$, Corollary \ref{cor:weaksharpextra} yields, for $\beta>-1$, $1\le p\le q<\infty$ with $\frac{1}{p}-\frac{1}{q}=1-\frac{1}{q_0}$ and $\frac{2+\beta}{q}=\frac{2+\gamma}{p}$, and for all $\omega\in B_{p,q,\gamma}$,
\begin{equation}\label{eq:sharweakend1}\|P_{\alpha}^+f\|_{L^{q,\infty}(\omega^{q}\,{d}V_{\beta})}\le c[\omega]_{B_{p,q,\gamma}}^{\frac{1}{p'}+\frac{1}{q}}\|f\|_{p,\omega^p,\gamma}.\end{equation}
This completes the proof.
\end{proof}

We now prove that the above estimate is sharp.
\begin{proof}[Proof of sharpness]
Put $u=\omega^q$; recall that $[\omega]_{B_{p,q,\gamma}}=[u]_{B_{1+\frac{q}{p'},\gamma}}$. Hence \eqref{eq:sharweakend1} is equivalent to
\begin{equation}\label{eq:sharweakend2}\|P_{\alpha}^+f\|_{L^{q,\infty}(u\,{d}V_{\beta})}\le c[u]_{B_{1+\frac{q}{p'},\gamma}}^{\frac{1}{p'}+\frac{1}{q}}\|f\|_{p,u^{p/q},\gamma}.\end{equation}
Since $\frac{1}{p}-\frac{1}{q}=\frac{1}{q_0'}$, \eqref{eq:sharweakend2} is equivalent to
\begin{equation}\label{eq:sharweakend3}\|P_{\alpha}^+(u^{\frac{1}{q_0'}}f)\|_{L^{q,\infty}(u\,{d}V_{\beta})}\le c[u]_{B_{1+\frac{q}{p'},\gamma}}^{\frac{1}{p'}+\frac{1}{q}}\|f\|_{p,u,\gamma}.\end{equation}
Since $[u]_{B_{1+\frac{q}{p'},\gamma}}\le [u]_{B_{1,\gamma}}$, assuming $u\in B_{1,\gamma}$ turns \eqref{eq:sharweakend3} into
\begin{equation}\label{eq:sharweakend4}\|P_{\alpha}^+(u^{\frac{1}{q_0'}}f)\|_{L^{q,\infty}(u\,{d}V_{\beta})}\le c[u]_{B_{1,\gamma}}^{\frac{1}{p'}+\frac{1}{q}}\|f\|_{p,u,\gamma}.\end{equation}
It therefore suffices to prove that \eqref{eq:sharweakend4} is sharp.

Let $1-\frac{1+\gamma}{1+\beta}<\delta<1$, and set \[
u(z)=|z|^{-s+\beta-\gamma}, \qquad s=(1+\beta)(1-\delta).
\]
Note that $-s+\beta>-1$. It is not hard to check that
\[
[u]_{B_{1,\gamma}}\approx \frac{1}{\delta}.
\]
Consider $f(z)=1_{\{z\in \mathbb{R}_+^2:|z|\leq 1\}}(z)=:1_{\mathbb{B}}$. Then
\[
\|f\|_{p,u,\gamma}\approx \delta^{-1/p}.
\]
Let $0<\eta=\eta(\delta)<1$ be a parameter to be fixed. We have
\begin{align*}
&\|P_{\alpha}^+(u^{\frac{1}{q_0'}}f)\|_{L^{q,\infty}(u\,{d}V_{\beta})}\\
&\ge \sup_{\lambda>0}\lambda|\{z\in \mathbb{R}_+^2\cap B(0,\eta):\int_{\mathbb{B}}\frac{|w|^{\frac{(1+\beta)\delta-\gamma-1}{q_0'}}}{|z-\bar{w}|^{2+\alpha}}dV_\alpha(w)>\lambda\}|_{u,\beta}^{1/q}\\
&\ge \sup_{\lambda>0}\lambda|\{z\in \mathbb{R}_+^2\cap B(0,\eta):\int_{\mathbb{B}\setminus B(0,|z|)}\frac{|w|^{\frac{(1+\beta)\delta}{q_0'}}}{(2|{w}|)^{2+\alpha}}dV_\alpha(w)>\lambda\}|_{u,\beta}^{1/q}\\
&= \sup_{\lambda>0}\lambda|\{z\in \mathbb{R}_+^2\cap B(0,\eta):\frac{q_0'}{2^{2+\alpha}(1+\beta)\delta}(1-|z|^{\frac{(1+\beta)\delta}{q_0'}})>\lambda\}|_{u,\beta}^{1/q}\\
&\ge \frac{q_0'}{2^{3+\alpha}(1+\beta)\delta}|\{z\in \mathbb{R}_+^2\cap B(0,\eta):\frac{q_0'}{2^{2+\alpha}(1+\beta)\delta}(1-|z|^{\frac{(1+\beta)\delta}{q_0'}})>\frac{q_0'}{2^{3+\alpha}(1+\beta)\delta}\}|_{u,\beta}^{1/q}\\
&= \frac{q_0'}{2^{3+\alpha}(1+\beta)\delta}|\{z\in \mathbb{R}_+^2\cap B(0,\eta):|z|<\left(\frac{1}{2}\right)^{\frac{q_0'}{(1+\beta)\delta}}\}|_{u,\beta}^{1/q}.
\end{align*}
Taking $\eta=\left(\frac{1}{2}\right)^{\frac{q_0'}{(1+\beta)\delta}}$, and using that $\beta-\gamma<(1+\beta)\delta$, we obtain
\begin{align*}
\|P_{\alpha}(u^{\frac{1}{q_0'}}f)\|_{L^{q,\infty}(u\,{d}V_{\beta})}
&\ge \frac{1}{2^{3+\alpha}(1+\beta)\delta}|\{z\in \mathbb{R}_+^2: |z|<\eta\}|_{u,\beta}^{1/q}\\
&\geq \frac{c}{2^{3+\alpha}(1+\beta)\delta}\left(\frac{\eta^{(1+\beta)\delta}}{(1+\beta)\delta}\right)^{1/q}\\
&= c\frac{1}{\delta}\left(\frac{1}{\delta}\right)^{1/q}\\
&= c\left(\frac{1}{\delta}\right)^{1/p}\left(\frac{1}{\delta}\right)^{\frac{1}{p'}+\frac{1}{q}}\\
&\approx [u]_{B_{1,\gamma}}^{\frac{1}{p'}+\frac{1}{q}}\|f\|_{p,u,\gamma}.
\end{align*}
This completes the proof.
\end{proof}
\subsection{Proof of Theorem \ref{thm:sharpweakbergnew}}
We begin with the Sawyer-type characterizations of the boundedness of the Bergman-type operator defined in \eqref{eq:fracBergdef}, as derived from \cite[Theorems 1 and 2]{Sehba6}.
\vskip .1cm
 We first record the following weak-type result, derived from \cite[Theorem 1]{Sehba6}.
\begin{proposition}\label{prop:sawyerweak}
Let $\alpha,\beta,\gamma>-1$, let $v$ and $u$ be positive locally integrable functions, and let $1< p\le q<\infty$. Then $P_{\alpha}^+:L^p(v\,\mathrm{d}V_{\gamma})\longrightarrow L^{q,\infty}(u\,\mathrm{d}V_{\beta})$ is bounded if and only if $u$ and $\sigma := v^{1-p'}$ satisfy
\[
[\sigma,u]_{q',p',\alpha,\beta,\gamma}:=\sup_{I\in \mathcal{I}}|Q_I|^{-\frac{1}{q'}}_{u,\beta}\|1_{Q_I}P_{\alpha}^+(1_{Q_I}y^{\beta-\alpha}u) \|_{p',\sigma,\alpha+\left(\gamma-\alpha\right)(1-p')}<\infty.
\]
Moreover,
\Be\label{eq:sawyerweak}\|P_{\alpha}^+\|_{L^p(v\,\mathrm{d}V_{\gamma})\longrightarrow L^{q,\infty}(u\,\mathrm{d}V_{\beta})} \approx [\sigma,u]_{q',p',\alpha,\beta,\gamma}.
\Ee
\end{proposition}

The following strong-type Sawyer result follows from \cite[Theorem 2]{Sehba6}.
\begin{proposition}\label{thm:sawyer}
Let $\alpha,\beta,\gamma>-1$, let $v$ and $u$ be positive locally integrable functions, and let $1<p\le q<\infty$. Then $P_{\alpha}^+:L^p(v\,\mathrm{d}V_{\gamma})\longrightarrow L^q(u\,\mathrm{d}V_{\beta})$ is bounded if and only if $u$ and $\sigma := v^{1-p'}$ satisfy
\[
[u,\sigma]_{S_{p,q,\alpha,\beta,\gamma}}:= \sup_{I\in \mathcal{I}}|Q_I|^{-\frac{1}{p}}_{\sigma,\alpha+\left(\gamma-\alpha\right)(1-p')}\|1_{Q_I}P_{\alpha}^+(1_{Q_I}y^{\left(\gamma-\alpha\right)(1-p')}\sigma) \|_{q,u,\beta}<\infty
\]
and
\[
[\sigma,u]_{q',p',\alpha,\beta,\gamma}:=\sup_{I\in \mathcal{I}}|Q_I|^{-\frac{1}{q'}}_{u,\beta}\|1_{Q_I}P_{\alpha}^+(1_{Q_I}y^{\beta-\alpha}u) \|_{p',\sigma,\alpha+\left(\gamma-\alpha\right)(1-p')}<\infty.
\]
Moreover
\Be\label{eq:sawyer}\|P_{\alpha}^+\|_{L^p(v\,\mathrm{d}V_{\gamma})\longrightarrow L^q(u\,\mathrm{d}V_{\beta})} \approx [u,\sigma]_{S_{p,q,\alpha,\beta,\gamma}} + [\sigma,u]_{S_{q',p',\alpha,\beta,\gamma}}.
\Ee
\end{proposition}

A direct check shows that
\[
\|P_{\alpha}^+\|_{L^{q'}(u^{1-q'}\,\mathrm{d}V_{\beta+(\alpha-\beta)q'})\longrightarrow L^{p',\infty}(v^{1-p'}\,\mathrm{d}V_{\alpha+(\gamma-\alpha)(1-p')})} \approx [\sigma,u]_{q,p,\alpha,\beta,\gamma}.
\]
This allows us to deduce that
\begin{equation}\label{eq:strongfromweak}
\begin{aligned}
\|P_{\alpha}^+\|_{L^p(v\,\mathrm{d}V_{\gamma})\to L^q(u\,\mathrm{d}V_{\beta})}
&\approx \|P_{\alpha}^+\|_{L^p(v\,\mathrm{d}V_{\gamma})\to L^{q,\infty}(u\,\mathrm{d}V_{\beta})}\\
&\quad + \|P_{\alpha}^+\|_{L^{q'}(u^{1-q'}\,\mathrm{d}V_{\beta+(\alpha-\beta)q'})\to L^{p',\infty}(v^{1-p'}\,\mathrm{d}V_{\alpha+(\gamma-\alpha)(1-p')})}.
\end{aligned}
\end{equation}
\vskip .3cm

We can now prove Theorem \ref{thm:sharpweakbergnew}.
\begin{proof}[Proof of Theorem \ref{thm:sharpweakbergnew}]
By \eqref{eq:sharpfracbergnew} and \eqref{eq:strongfromweak},
\begin{align*}
\|P_{\alpha}^+\|_{L^p(\omega^p\,\mathrm{d}V_{\gamma})\longrightarrow L^q(\omega^q\,\mathrm{d}V_{\beta})}
&\approx \|P_{\alpha}^+\|_{L^p(\omega^p\,\mathrm{d}V_{\gamma})\longrightarrow L^{q,\infty}(\omega^q\,\mathrm{d}V_{\beta})}\\
&\quad + \|P_{\alpha}^+\|_{L^{q'}(\omega^{-q'}\,\mathrm{d}V_{\beta+(\alpha-\beta)q'})\longrightarrow L^{p',\infty}(\omega^{-p'}\,\mathrm{d}V_{\alpha+(\gamma-\alpha)(1-p')})}\\
&\lesssim [\omega]_{B_{p,q,\gamma}}^{\frac{1}{p'}+\frac{1}{q}}+[\omega^{-1}]_{B_{q',p',\beta+(\alpha-\beta)q'}}^{\frac{1}{p'}+\frac{1}{q}}\\
&\lesssim [\omega]_{B_{p,q,\gamma}}^{\frac{1}{p'}+\frac{1}{q}}+[\omega]_{B_{p,q,\beta+(\alpha-\beta)q'}}^{\left(\frac{1}{p'}+\frac{1}{q}\right)\frac{p'}{q}}
\end{align*}
where we used  $2+\alpha+(\gamma-\alpha)(1-p')=\left[2+\beta+(\alpha-\beta)q'\right]\frac{p'}{q'}$, $\alpha\ge\beta+(\alpha-\beta)q'$ and Theorem \ref{thm:sharpfracbergnew}, and also $[\omega^{-1}]_{B_{q',p',\lambda}}=[\omega]_{B_{p,q,\lambda}}^{\frac{p'}{q}}$. This completes the proof.
\end{proof}

 \section{Conclusion}

We have established the sharp weak-type off-diagonal estimate and derived a strong-type off-diagonal bound for the positive Bergman operator \(P_\alpha^+\) in the two-parameter setting, allowing the operator exponent \(\alpha\) and the domain exponent \(\gamma\) to be genuinely distinct. The weak-type  off-diagonal estimate obtained in Theorem \ref{thm:sharpfracbergnew} is sharp and provides the first fully decoupled treatment of this problem. The strong-type estimate in Theorem \ref{thm:sharpweakbergnew} does not reduce to a single characteristic in general. However, when the two characteristics coincide, i.e. when
$\gamma=\beta+(\alpha-\beta)q'$, their sum is comparable to the single
characteristic
\[
[\omega]_{B_{p,q,\gamma}}^{\left(\frac{1}{p'}+\frac{1}{q}\right)\max\{1,\frac{p'}{q}\}}.
\]
Under the balance condition
$\frac{2+\alpha}{2+\gamma}=\frac{1}{p}+\frac{1}{q'}$ of
Corollary \ref{cor:sharpstrong}, the two characteristics always coincide,
and $\alpha=\gamma$ is equivalent to $p=q$. More generally, the
off-diagonal estimate (\ref{eq:strongberg}) is not recovered from
Corollary \ref{cor:sharpstrong} except in its diagonal case $p=q$;
this is the intriguing feature of the decoupling.

The central innovation is the off-diagonal extrapolation principle, Theorem \ref{thm:sharpextra}, which permits the power of the distance to the boundary to vary with the exponents, thereby removing the coupling inherent in previous fractional-operator approaches. The fact that the weight characteristic \([\omega]_{B_{p,q,\gamma}}\) depends exclusively on the domain parameter \(\gamma\), while the operator parameter \(\alpha\) appears only in the implicit constants, highlights the flexibility of our method. We expect the techniques developed here to be useful in other contexts where a change in the ambient measure parameter is required.

\end{document}